\documentclass[12pt,letterpaper,oneside]{amsart}

 \usepackage{amsmath,amssymb,amsthm}
 \usepackage{amscd}
 \usepackage{geometry}
 \usepackage{graphicx,epstopdf}
 \usepackage{color}
 \usepackage{enumerate}
 \usepackage{xspace}
 \usepackage{tipa}
 \usepackage{stmaryrd}
 \usepackage{labelfig}
 \usepackage{epsfig}
 \usepackage{pinlabel}
 \usepackage[all]{xy}

 \usepackage{hyperref}

 \newcommand{\calA}{\mathcal{A}}

 \newcommand{\calH}{\mathcal{H}}

 \newtheorem{theorem}{Theorem}[section]
 \newtheorem{proof of the main theorem}[theorem]{Proof of the Main Theorem}
 \newtheorem{proposition}[theorem]{Proposition}
 \newtheorem{observation}[theorem]{Observation}

 \newtheorem{lemma}[theorem]{Lemma}
 
 \newtheorem*{conjecture*}{Conjecture}

 \theoremstyle{definition}
 \newtheorem{definition}[theorem]{Definition}

 \newtheorem*{claim*}{Claim}

 \newtheorem*{question*}{Question}
 \newtheorem*{answer*}{Answer}
 \newtheorem*{application*}{Application}
 \newtheorem*{ideas*}{ideas}
 \theoremstyle{remark}
 \newtheorem{remark}[theorem]{Remark}
 \newtheorem*{remark*}{Remark}

 \DeclareMathOperator{\Mod}{Mod}

 \newcommand{\param}{{\mathchoice{\mkern1mu\mbox{\raise2.2pt\hbox{$
 \centerdot$}}
 \mkern1mu}{\mkern1mu\mbox{\raise2.2pt\hbox{$\centerdot$}}\mkern1mu}{
 \mkern1.5mu\centerdot\mkern1.5mu}{\mkern1.5mu\centerdot\mkern1.5mu}}}
\renewcommand{\setminus}{{\smallsetminus}}

 \newcommand{\Z}{{\mathbb{Z}}}

\newcommand{\p}{{(\Gamma, \mathcal{A})}}

\newcommand{\A}{\mathcal{A}}
\newcommand{\hX}{{\mathcal{H}(X)}}
\newcommand{\hS}{{\mathcal{H}(\Sigma)}}

\newcommand{\pS}{{\mathcal{P}(\Sigma)}}

\begin{document}


\title {The coarse geometry of hexagon decomposition graphs}

 \author{Funda G\"ultepe}
 \thanks{The first author was partially supported by NSF grant DMS-2137611. Both authors were partially supported by ANR/FNR project SoS, INTER/ANR/16/11554412/SoS, ANR-17-CE40-0033.}
 \address{Department of Mathematics and Statistics \\
 University of Toledo\\
 Toledo, OH, 43606}
\email{funda.gultepe@utoledo.edu}
\urladdr{http://www.math.utoledo.edu/~fgultepe/}
\author {Hugo Parlier}

\address{Department of Mathematics\\
University of Luxembourg\\
Esch-sur-Alzette, Luxembourg}
\email{hugo.parlier@uni.lu}
\urladdr{http://math.uni.lu/parlier/}

\begin{abstract}
We define and study graphs associated to hexagon decompositions of surfaces by curves and arcs. One of the variants is shown to be quasi-isometric to the pants graph, whereas the other variant is quasi-isometric to (a Cayley graph of) the mapping class group.
\end{abstract}

\maketitle
\section{Introduction}
Simplicial complexes related to curves and arcs have been used with considerable success in the study of surfaces and their deformation spaces. In particular, they have proved to be efficient tools for the study of the large scale geometry of Teichm\"uller spaces and mapping class groups. For instance, curve graphs are rough models for the electrified Teichm\"uller metric, as is the marking graph for mapping class groups \cite{Masur-MinskyI,Masur-MinskyII}, and the pants graph for the Weil-Petersson metric \cite{Brock}. Flip-graphs, which originally appeared in topics closer to computational or combinatorial geometry, also provide quasi-models for mapping class groups, but, like arc graphs, they require marked points to serve as basepoints for arcs. This is also true of other similar graphs, such as the polygonalisation graph introduced recently \cite{BDT}.

Here we introduce graphs that use both curves and arcs, but unlike standard arc graphs, do not require marked points. Instead, homotopy classes of arcs go between curves (with endpoints allowed to glide on the curves). In particular, these graphs work for surfaces with or without marked points, and so work for closed surfaces in particular, while mimicking some of the nice properties of flip-graphs. Roughly speaking, the vertices of the graphs are collections of curves and arcs that decompose the surface into hexagons. Edge relations come from either flipping an arc, or from adding or removing a curve. We have two versions, depending on whether we put weights on the curves.

On an orientable, finite-type and of negative Euler characteristic surface $\Sigma$, our first graph $\hS$ is defined purely topologically (see Section \ref{sec:prelim} for a precise definition) and hence we call it the \emph{topological hexagon decomposition graph}. We view it as a metric space by giving each edge unit length. Our main result is that this graph is quasi-isometric to the standard pants graph $\mathcal{P}(\Sigma)$.

\begin{theorem}\label{thm:intropantsqi}
For finite-type orientable surfaces, the topological hexagon decomposition graph $\hS$ is quasi-isometric to the pants graph $\mathcal{P}(\Sigma)$.
\end{theorem}

By Brock's results on the pants graph \cite{Brock}, this means that $\hS$ is also a coarse model for the Weil-Petersson metric.

Our second type of graph is an augmented version of the above graph, where vertices are the same but with a collection of weights on each curve. Edges come from elementary moves as for $\hS$, but there is the additional consideration of how to determine the weight of an added curve. There would be multiple ways to do this, but we choose to associate weights by measuring lengths with a fixed hyperbolic surface $X$. We call the resulting graph the \emph{geometric hexagon decomposition graph} and denote it $\hX$. In contrast to $\hS$, vertices in $\hX$ are of finite valency (Theorem \ref{thm:valency}), and there is a natural (coarse) action of the mapping class group $\Mod(X)$ (which is also $\Mod(\Sigma)$) on it. In contrast to the action of $\Mod(\Sigma)$ on $\hS$, stabilizers are now finite, and hence we obtain:

\begin{theorem}\label{thm:intromcgqi}
For finite-type orientable surfaces, the geometric hexagon decomposition graph $\hX$ is quasi-isometric to the mapping class group $\Mod(X)$.
\end{theorem}

By work of Masur and Minsky \cite{Masur-MinskyII}, the geometric hexagon graph is also quasi-isometric to the \emph{marking graph}. The proof of this theorem follows from a version of the Milnor-Schwartz lemma for quasi-actions. Most of the work goes in to showing the graph has the properties that allow this to work. It is interesting to note that the topological graph, which requires only topological input, provides a coarse model for deformation spaces of hyperbolic surfaces, whereas the geometric graph provides a model for $\Mod(\Sigma)$.

{\bf Organization.} In a preliminary part (Section \ref{sec:prelim}), we define the main building blocks we need, namely hexagon decompositions with and without weights. In Section \ref{sec:topo} we define and study the topological hexagon graph, and show it is a coarsely equivalent to the pants graph. In Section \ref{sec:ghg}, we study the geometric version and show it is a coarse model for the mapping class group.

{\bf Acknowledgements.} We thank Mark Bell, Mark Hagen, Chris Leininger, Pallavi Panda and Binbin Xu for fruitful discussions, comments and suggestions.

\section{Preliminaries}\label{sec:prelim}

Let $\Sigma$ be a finite type orientable surface with all connected components of negative Euler characteristic. For reasons of clarity, we think of boundary as consisting of curves (and not punctures or marked points). A {\it curve} will be an abbreviation for an isotopy class of essential simple closed curve. If it is non-peripheral to boundary it is an interior curve. On a surface with boundary, by {\it arc} we will mean a simple arc between boundary elements up to homotopy with endpoints gliding on the boundary.

We will also need to consider multicurves which are collections of disjoint and distinct curves. Note that there are different uses of the word multicurve in the literature but as all of our curves are simple, we avoid calling them systematically simple multicurves. Similarly, given a surface with boundary, a multiarc is a collection of distinct and disjoint arcs.

A multicurve or multiarc is said to be maximal if it is maximal with respect to inclusion. Note that a maximal multicurve is a pants decomposition, and a maximal multiarc decomposes the surface into hexagons. (Note that it is for exactly this reason that we prefer boundary curves to punctures or marked points: otherwise we would have (ideal) polygons of different sizes in the complementary regions of our collection of arcs.)

We denote by $\kappa_{\mathrm{c}}(\Sigma)$ the curve complexity of $\Sigma$, that is the number of curves in a pants decomposition. Similarly, we denote by $\kappa_{\mathrm{a}}(\Sigma)$ the arc complexity, that is the number of arcs in a hexagon decomposition. For example if $\Sigma$ is closed and of genus $g$, we have $\kappa_{\mathrm{c}}(\Sigma)= 3g-3$ and $\kappa_{\mathrm{a}}(\Sigma\setminus \Gamma)= 4g-4$ for any non-empty multicurve $\Gamma$.

A hexagon decomposition is a pair $(\Gamma,\A)$ where $\Gamma$ is a non-empty multicurve and $\A$ is a maximal multiarc on $\Sigma\setminus \Gamma$. In the case where $\Sigma$ has boundary, we require that $\Gamma$ contain all peripheral curves of $\Sigma$.

A weighted multicurve $(\Gamma, w)$ is a multicurve $\Gamma=\bigcup_{k=1}^{|\Gamma|} \gamma_k$ together with a collection of weights $w=w(\Gamma) = \{k_1,\cdots, k_{|\Gamma|}\}\in \Z^{|\Gamma|}$. Note we only allow weights to be integers.

It will be useful to consider metric surfaces: for given $\Sigma$ we consider a fixed hyperbolic surface $X$ homeomorphic to $\Sigma$, and if $\Sigma$ has boundary, we ask that the boundary elements of $X$ be realized by simple closed geodesics. (This is simply because we want to be able to measure the lengths of arcs in the easiest possible way.) Thus on $X$, curves are uniquely realized by simple closed geodesics and arcs are uniquely realized by simple orthogeodesics (geodesics orthogonal to their terminal simple closed geodesics). A maximal multiarc is then realized by a maximal collection of disjoint orthogeodesics and the complementary region is a collection of right-angled hexagons.

Before introducing the hexagonal graphs, we recall that the pants graph $\mathcal{P}(\Sigma)$ is the graph where vertices are pants decompositions and two vertices are related by an edge if they differ by an elementary move. The two types of elementary moves are illustrated in figures \ref{fig:elementary1} and \ref{fig:elementary2}.

\begin{figure}[htbp]
\begin{center}
\includegraphics[width=6cm]{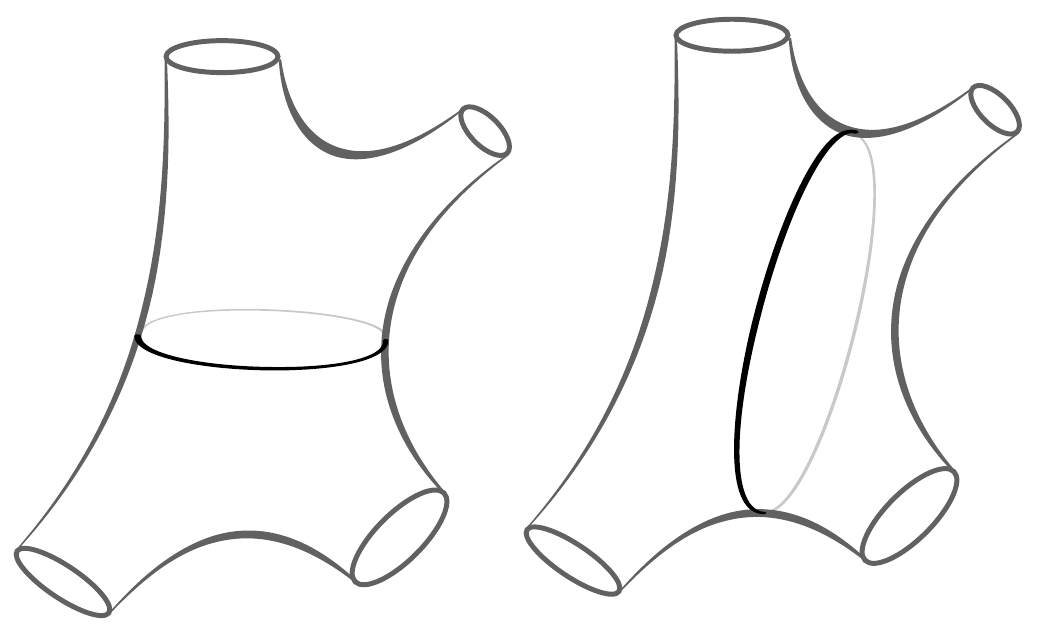}
\caption{The first type of elementary move is on a four-holed sphere}
\label{fig:elementary1}
\end{center}
\end{figure}

\begin{figure}[htbp]
\begin{center}
\includegraphics[width=6cm]{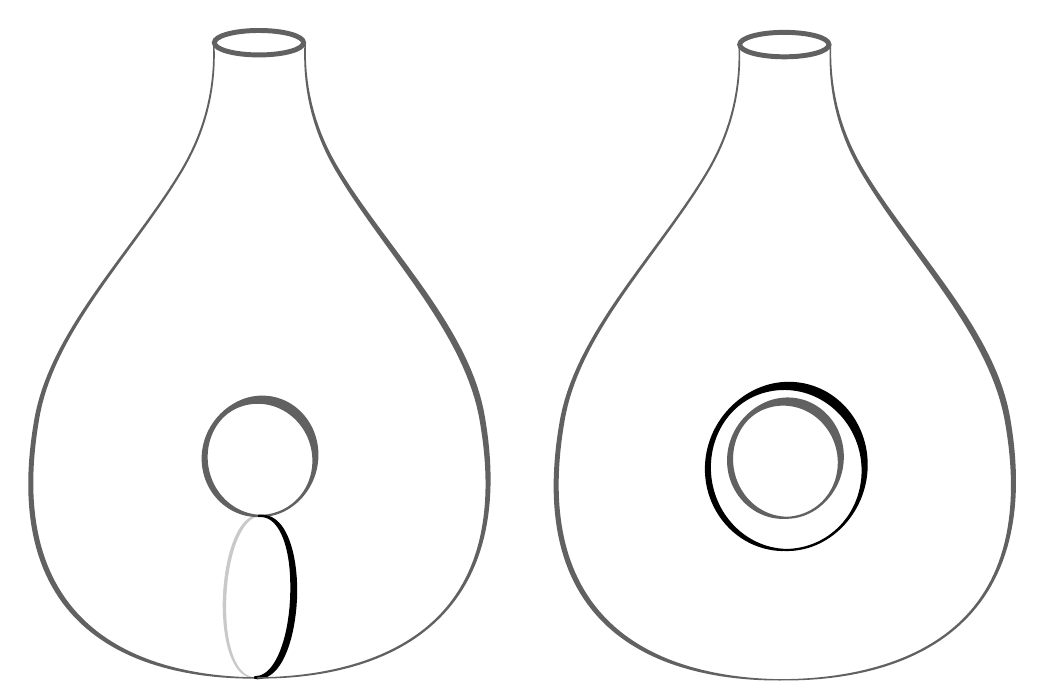}
\caption{The second type of elementary move is on a one-holed torus}
\label{fig:elementary2}
\end{center}
\end{figure}

Brock's result \cite{Brock} that the pants graph is quasi-isometric to Teichm\"uller space with the Weil-Petersson metric is one example among many of how to use a combinatorial model to study moduli type spaces and their isometry groups. We refer the reader to \cite{Behr, Brock-Farb, Masur-Sch-disk, Masur-Sch-pants, Brock-Margalit, Brock-Masur, Hat-Thu,Disarlo-Parlier, Gultepe-Leininger, Hatcher-triangulation, Mosher} for others.  

\section{The topological graph and its properties}\label{sec:topo}

\subsection{The definition}

The topological version only depends on $\Sigma$ and does not depend on a choice of metric, hence we denote it $\hS$.

Vertices will be given by hexagon decompositions and hexagon decompositions are related by an edge if they satisfy one of the two relations. The first one is a type of flip relation between maximal multiarcs.

\noindent {\it The flip relation.} $H=(\Gamma, \A)$ and $H'= (\Gamma', \A')$ are related by a flip if $\Gamma=\Gamma'$ and $\A$ and $\A'$ differ by a single arc. This is analogous to flip operations between triangulations. See Figure \ref{fig:flip} for an illustration.

\begin{figure}[htbp]
\begin{center}
\includegraphics[width=6cm]{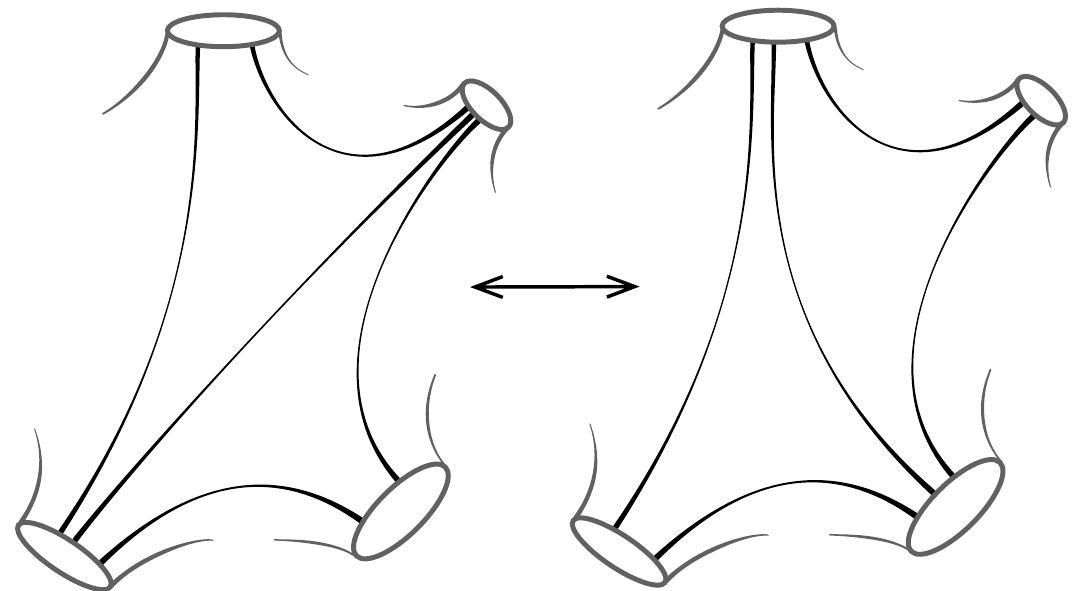}
\caption{The illustration of a flip. Note that the curves of $\Gamma$ the arcs terminate on need not be distinct.}
\label{fig:flip}
\end{center}
\end{figure}

The second one relates to transformations of the multiarc.

\noindent {\it Curve addition.}

Let $H=(\Gamma, \A)$ be a hexagon decomposition with $\Gamma$ non-maximal. A curve $\alpha$ is said to be compatible with $H$ if $\alpha \cap \Gamma= \emptyset$ and $i(a,\alpha)\leq 1$ for all $a\in \A$.

For $\alpha$ compatible with $H$, we obtain a new hexagon decomposition $H' = (\Gamma', \A')$ as follows. We set $\Gamma'= \Gamma\cup\alpha$. Now we consider $\A'$ the collection of arcs obtained from $\A$ which are either disjoint from $\alpha$ or which are obtained by surgering an arc of $\A$ at its intersection point with $\alpha$. The next lemma shows that $\A'$ is maximal, and hence $H'$ is a hexagon decomposition.

\begin{lemma}
The multiarc $\A'$ is maximal on $\Sigma \setminus \Gamma'$
\end{lemma}
\begin{proof}
The complementary region of $H$ on $\Sigma$ is a collection of hexagons. The curve $\alpha$, being disjoint from $\Gamma$, crosses these hexagons in the sides corresponding to arcs of $\A$. The complementary regions to $\alpha$ on each hexagon are thus a collection of quadrilaterals and hexagons. The quadrilaterals correspond to a pair of subarcs of $\A$ that will be freely homotopic on $\Sigma\setminus \Gamma'$ where $\Gamma' = \Gamma \cup \alpha$, and so the resulting mutliarc decomposes $\Sigma\setminus \Gamma'$ into hexagons.
\end{proof}

We say that $H=(\Gamma, \A)$ and $H'= (\Gamma', \A')$ are related by a curve addition if $H'$ is obtained by adding a compatible curve $\alpha$ as described above.

\begin{definition} The \emph{topological} hexagon decomposition graph $\hS$ is the graph whose vertex set is given by hexagon decompositions and $H,H'$ are joined by an edge if they are either related by a flip or a curve addition.
\end{definition}

\begin{remark}
The operation of curve addition leads naturally to the notion of curve removal, which is just the formal opposite operation, meaning that $H$ can be obtained from $H'=(\Gamma',\A')$ by curve removal if $H'$ is obtained by a curve addition to $H$. However, in contrast with curve addition, the removal of a curve $\alpha\in \Gamma'$ is not uniquely defined by a choice of curve $\alpha$. Indeed, consider $\alpha$ compatible with $H$, and $\tilde{H}$ obtained by Dehn twisting $H$ along $\alpha$. Observe that $\alpha$ is also compatible with $\tilde{H}$. If adding $\alpha$ to $H$ results in $H'$, then adding $\alpha$ to $\tilde{H}$ also results in $H'$. In particular this means that $\hS$ contains vertices of infinite valency and finite valency. Infinite valency comes from curve removal, and certain curves are not removable. As such, finite valency corresponds to hexagon decompositions without any removable curves. \end{remark}

\subsection{Connectivity}

The first thing we show about these graphs is their connectivity.

For this a useful quantity will be the following. For any multicurve $\Gamma$, the subsurface $\Sigma'\subset \Sigma$ obtained by cutting $\Sigma$ along $\Gamma$ has an associated flip-graph, which is the subgraph of $ \calH(\Sigma)$ consisting in all vertices of the form $(\Gamma,\A)$, and edges coming from flip relations between them. By standard results on flip-graphs, this subgraph is connected. Furthermore, we can take its quotient by homeomorphisms that fix curves in $\Gamma$ individually, to obtain a finite graph. The diameter of this modular subgraph is bounded by a constant that only depends on the topology of $\Sigma'$ and in fact on its arc complexity. Note that this only depends on the homeomorphism type of $\Gamma$. We can thus take the maximum of these diameters among all homeomorphism types of (non-empty) $\Gamma$. We denote this quantity $D=D(\Sigma)$. It is a useful quantity because if it means that from any given hexagon decomposition, in at most $D$ moves, we can choose the topological type of arcs.

\begin{lemma}\label{lem:gotopants}

Any $(\Gamma,\A)\in \calH(\Sigma)$ is connected to a $(P, \A_P)\in \calH(\Sigma)$ where $P$ is a pants decomposition and $\Gamma\subset P$ in at most $K$ moves where $K$ is a constant that only depends on topology.
\end{lemma}

\begin{proof}
The basic idea of the proof is to add curves one by one to $\Gamma$ until we reach a full pants decomposition. To do so might require performing flip moves ahead of time.

If $\Gamma $ is not a pants decomposition, there exists a connected component $\Sigma'\subset \Sigma \backslash \Gamma$ which is not a pair of pants. We focus our attention on this subsurface, and show how to add a curve that lies inside $\Sigma'$ to our multicurve $\Gamma$.

Observe that any such subsurface contains an embedded cylinder not peripheral to one of its boundary components. This cylinder can be bounded by two arcs and subarcs of boundary components.

\begin{figure}[htbp]
\begin{center}
\includegraphics[width=8cm]{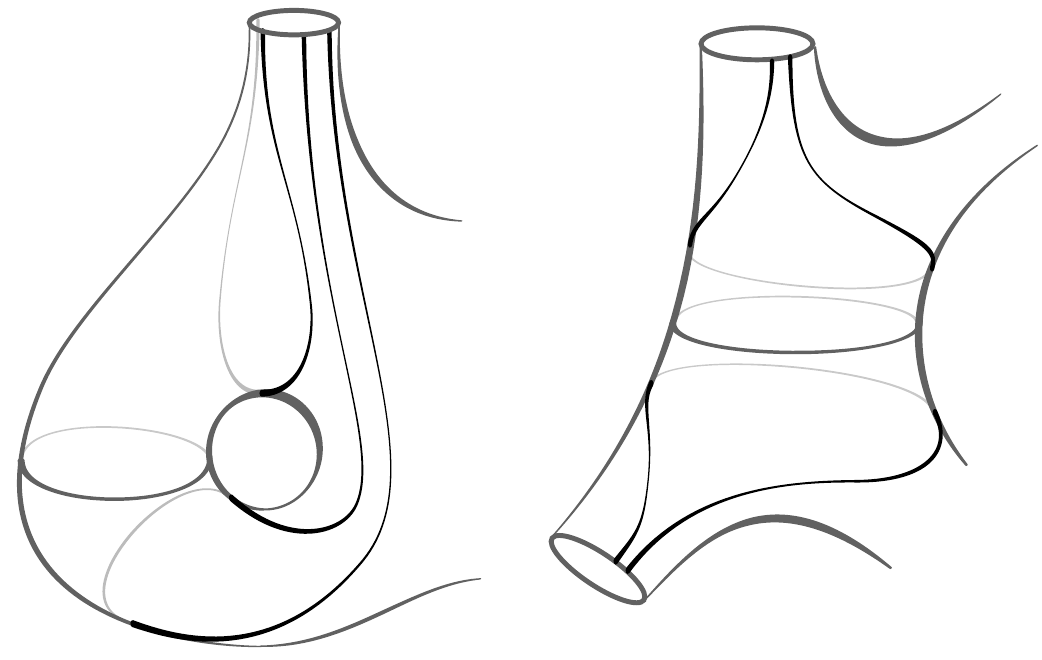}
\caption{Arcs bounding embedded cylinders in the non-planar and planar cases}
\label{fig:cylinder1}
\end{center}
\end{figure}

Figure \ref{fig:cylinder1} portrays the two possible cases, depending on whether the surface is planar or not. We add arcs to the cylinder subsurface as in Figure \ref{fig:cylinder2}.

By adding arcs, this can be completed into a hexagon decomposition. The core curve of the cylinder is now addable.

\begin{figure}[htbp]
\begin{center}
\includegraphics[width=8cm]{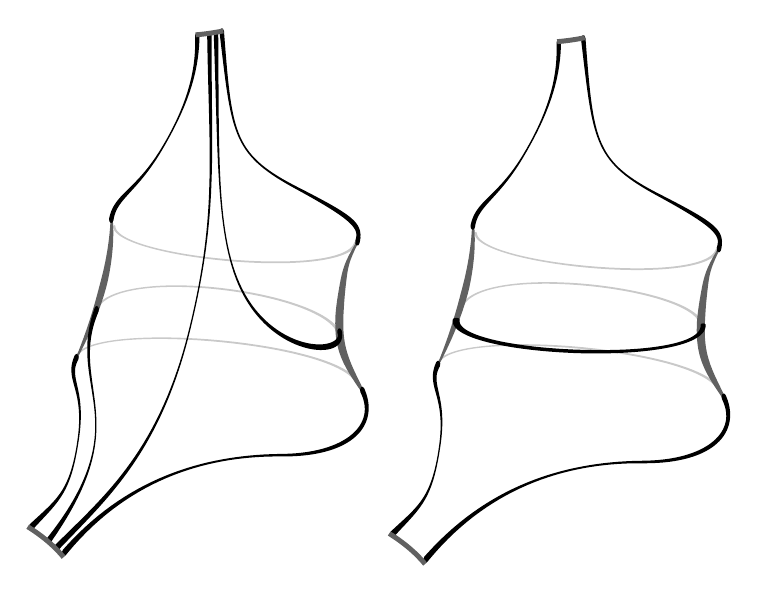}
\caption{A hexagon decomposition of the cylinder and its core curve}
\label{fig:cylinder2}
\end{center}
\end{figure}

Note that, as we are only interested in this configuration up to homeomorphism, such a hexagon decomposition can be reached in at most $D$ moves where $D$ is the constant described above. Hence, a curve has been added in $D+1$ moves. As there are at most $\kappa(\Sigma)-1$ curves to add, we reach a full pants decomposition in at most $(D+1) (\kappa(\Sigma)-1)$ moves.
\end{proof}

We now observe that moves in the pants graph can be emulated by moves in $\hS$. Consider two pants decompositions in $\pS$ that differ by an elementary move. To each we consider a hexagon decomposition by adding arcs, and such the hexagon decompositions only differ in the subsurface where the elementary move takes place. This is either a four-holed sphere or a one-holed torus. We add arcs to the corresponding subsurface as in figures \ref{fig:elementary1arcs} and \ref{fig:elementary2arcs}.

\begin{figure}[htbp]
\begin{center}
\includegraphics[width=10cm]{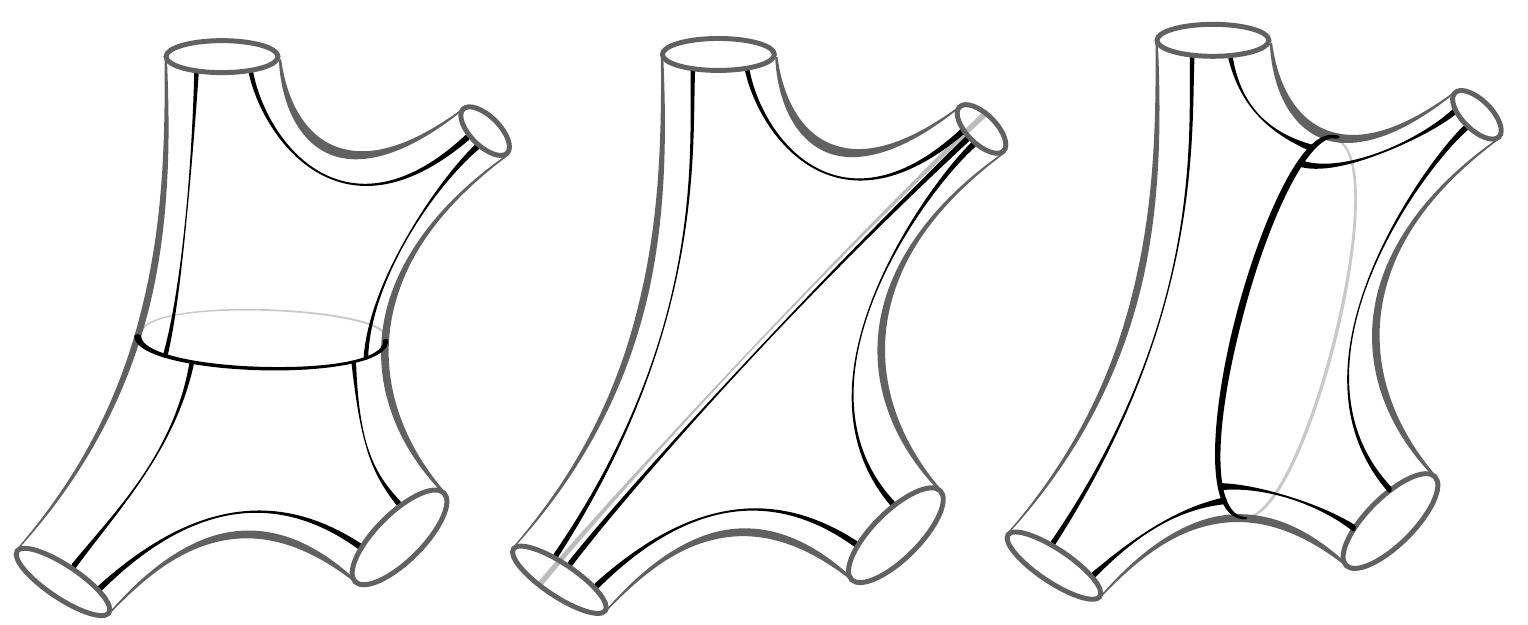}
\caption{The first type of elementary move with hexagons}
\label{fig:elementary1arcs}
\end{center}
\end{figure}

\begin{figure}[htbp]
\begin{center}
\includegraphics[width=10cm]{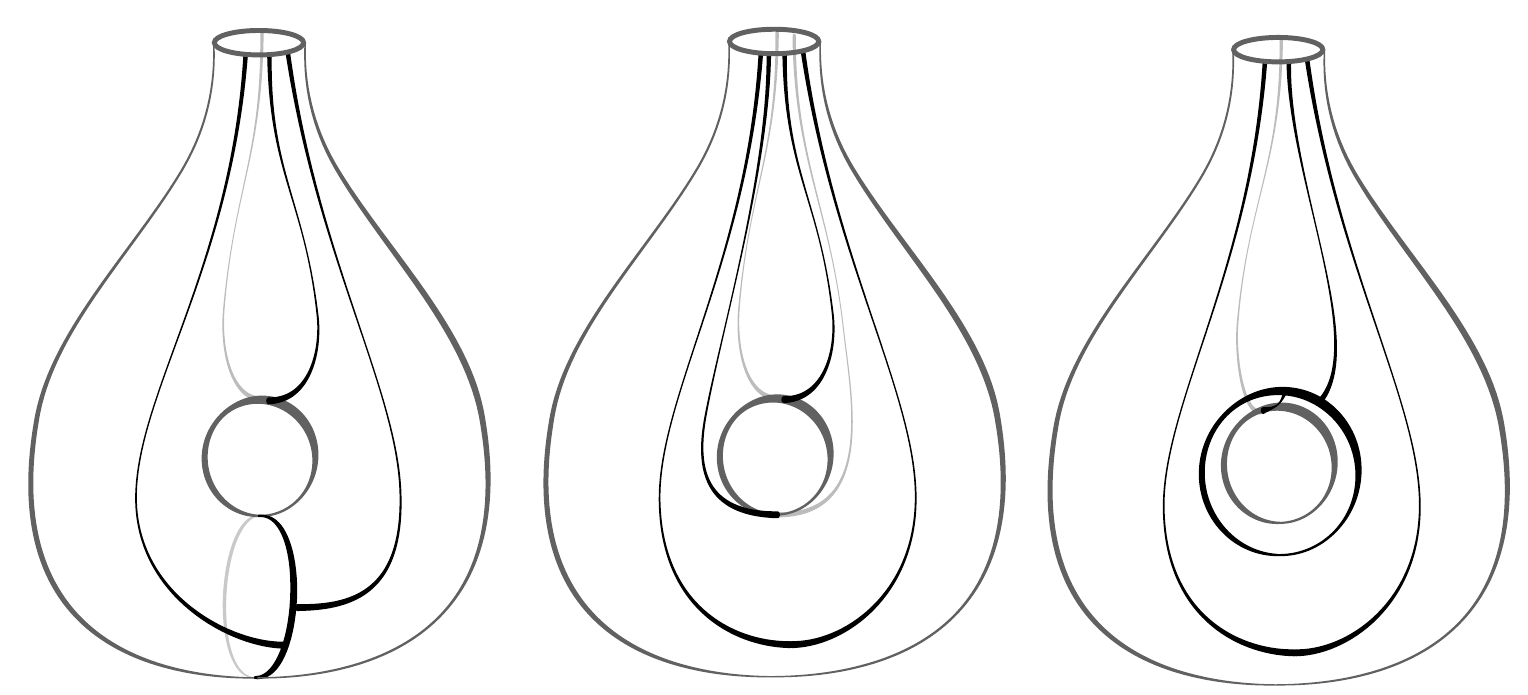}
\caption{The second type of elementary move with hexagons}
\label{fig:elementary2arcs}
\end{center}
\end{figure}

Now we can perform the moves illustrated in the figure to emulate the elementary move and deduce connectedness.

\begin{theorem}\label{thm:topoconnect}$\hS$ is connected.
\end{theorem}
\begin{proof}
By Lemma \ref{lem:gotopants} any vertex of $\hS$ is connected to a hexagon decomposition with a full pants decomposition. As elementary moves in the pants graph can be emulated in $\hS$, connectivity of $\hS$ then follows from the connectivity of the pants graph.
\end{proof}

We end this subsection with a basic result about curve addition. It is a consequence of the following observation which is a standard fact about normal coordinates for triangulations.

\begin{observation}
Let $\p \in \hS$. Let $\gamma$ be a curve disjoint from $\Gamma$. Then $\gamma$ is uniquely determined by its intersection with arcs in $\A$.
\end{observation}

Note that it is also true for any simple multicurve.

\begin{lemma}\label{lem:curveadding1}
For any $\p\in \hS$, there are at most $K$ curves that can be added to $\p$ where $K$ depends only on the topology of $\Sigma$.
\end{lemma}

\begin{proof}Given $\p$, only the compatible curves can be added to $\Gamma$. By definition, we may add only the curves which intersect $\calA$ at most once, hence a compatible curve can intersect each triangle in a triangulation at most twice. Hence there are at most $2^{|\calA|}=2^{\kappa_{\mathrm{a}}}$ ways to add a curve to $\Gamma$.
\end{proof}

Note the strong contrast to curve removal.

\subsection{A quasi-isometry with the pants graph}

In this subsection, we prove the following.

\begin{theorem}\label{thm:quasipants}
The graphs $\hS$ and $\mathcal{P}(\Sigma)$ are quasi-isometric.
\end{theorem}

Before passing to the proof, we describe two maps between $\hS$ and $\mathcal{P}(\Sigma)$.

\underline{The map $\phi:\pS \to \hS$}

We define the map $\phi:\pS \to \hS$ as follows. Given $P\in \pS$, we add a collection of arcs to obtain an element of $\hS$. That is: $\phi(P)=(P,\calA)$. The choice of $\calA$ is arbitrary, but notice that any two choices are a bounded distance apart, and in fact are at most distance $D$ apart where $D$ is the diameter constant defined previously. As an example of a choice of $\calA$ one can add arcs to each pair of pants as in Figure \ref{fig:pantstohex}.

\begin{figure}[htbp]
\begin{center}
\includegraphics[width=7cm]{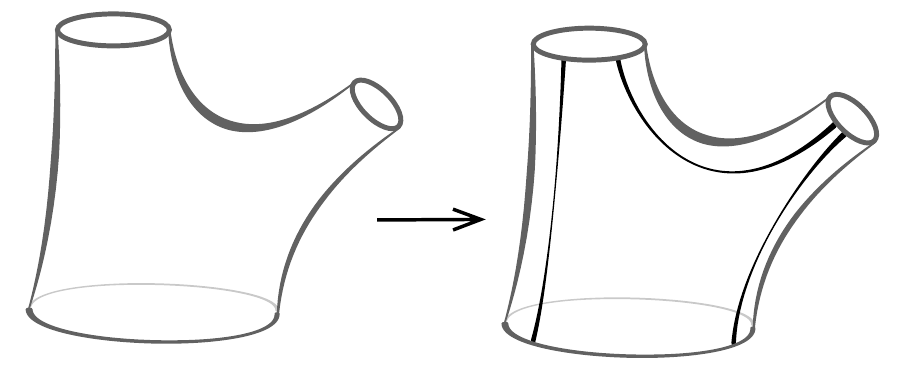}
\caption{A local picture of the map $\phi$}
\label{fig:pantstohex}
\end{center}
\end{figure}

\underline{The map $\psi:\hS \to \pS$}
This second map is less obvious and takes a hexagon decomposition and associates to it a pants decomposition. To define it, we shall use the strategy of Lemma \ref{lem:gotopants}, that is to add curves until the multicurve is a full pants decomposition. It is somewhat loosely defined, but this is ok as we will only need possible images to be at a bounded distance one from another.

Recall the strategy: Given $H=(\Gamma,\A) \in \hS$, where $\Gamma$ is not a pants decomposition, we add a curve to $\Gamma$. If it is not possible to add a curve, as shown in the proof of Lemma \ref{lem:gotopants}, by performing a minimal number of flips, we can add a curve, and then we repeat. Note that the number of necessary flips before adding a curve is at most $D$. The whole process ends in at most $(D+1) (\kappa(\Sigma)-1)$ steps. We define $\psi(H)$ to be the pants decomposition thus obtained. Observe that $\psi(\phi(P))=P$.

We now proceed that these maps satisfy certain properties that will lead us to the conclusion of Theorem \ref{thm:qipants}. We denote by $d_{\mathcal{H}}$ distance in $\hS$, and by $d_{\mathcal{P}}$ distance in $\pS$.

\begin{lemma}\label{lem:phidist}
There exists a constant $C_1=C_1(\Sigma)$ that depends only on the topology of $\Sigma$ such that for all $P,P'\in \pS$ we have
$$
d_{\mathcal{H}}(\phi(P), \phi(P')) \leq C_1\, d_{\mathcal{P}}(P,P').
$$
\end{lemma}

\begin{proof}
As explained previously, moves in $\pS$ can be emulated by moves in $\hS$ (see figures \ref{fig:elementary1arcs} and \ref{fig:elementary2arcs}), hence the result.
\end{proof}

Note that, if $\phi$ is defined exactly as in Figure \ref{fig:pantstohex}, then $C_1$ above can be taken to be $2$. Otherwise, it might require performing flips inside each pair of pants beforehand.

\begin{lemma} There exists a constant $K=K(\Sigma)$, which only depends on the topology of $\Sigma$, such that the full image $\phi(\pS)$ of the map $\phi$ is $K$-dense in $\hS$.
\end{lemma}
\begin{proof}
This follows from the proof of connectivity where we showed that any hexagon decomposition can be transformed into a hexagon decomposition with a full pants decomposition $P$ in a at most $(D+1) (\kappa(\Sigma)-1)$ steps. The map $\phi$ has in its image exactly one hexagon decomposition with a given pair of pants decomposition $P$. We now observe that any two hexagon decompositions with the same underlying multicurve (in this case $P$) are at most distance $D$ apart, which proves the claim.
\end{proof}

We now focus on the map $\psi$. The key step will be to show that it is ``quasi" well-defined, that is that any two possible choices of image are close to each other.

Given $H=(\Gamma_0,\A_0)$, we perform at most $D$ flips on $\A_0$ to obtain a new collection of arcs, say $\A'_0$. We then add a curve, say $\alpha_1$. Then because $\A_0$ and $\A'_0$ are related by at most $D$ flips, the (total) intersection between $\A_0$ and $\A'_0$ is bounded above by a function of the topology of $\Sigma$, say $I_1$. An explicit value for $I_1$ can be deduced from \cite[Corollary 2.18]{Disarlo-Parlier}.

We now make the following useful observation:

\underline{Observation:} As $\alpha_1$ intersects each arc of $\A'_0$ at most $1$ time, it also intersects $\A_0$ at most $I_1$ times.

To show this, note that $\alpha_1$ can be thought of as a concatenation of segments passing through the hexagons bounded by $\A'_0$ (at most one per hexagon). Now the homotopy class of each segment can be represented by a concatenation of at most 3 segments lying on the boundary of a hexagon (see Figure \ref{fig:hexsegment}).

\begin{figure}[htbp]
\begin{center}
\includegraphics[width=6cm]{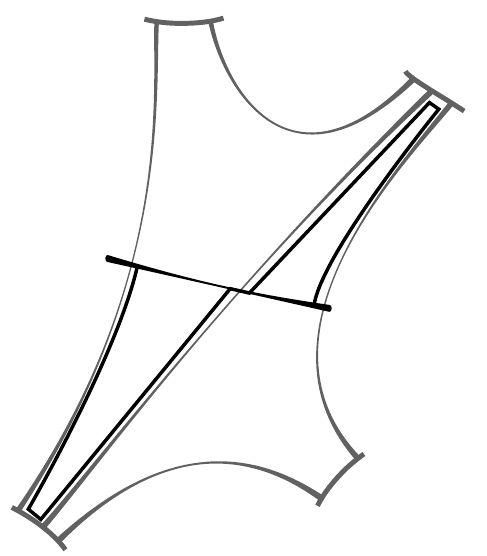}
\caption{An arc of $\alpha_1$ crossing two adjacent hexagons and represented by a broken path}
\label{fig:hexsegment}
\end{center}
\end{figure}

Then, by retracting back and forth segments if necessary, $\alpha_1$ can be represented by a concatenation of arcs of $\A_1$ and boundary segments of $\Gamma$, where each arc of $\A'_0$ appears at most one time. From this we can deduce that the total intersection between $\alpha_1$ and $\A_0$ is at most that of $\A'_0$ and $\A_0$, hence is at most $I_1$. We now add $\alpha_1$ to obtain a new hexagon decomposition, $\Gamma_1, \A_1$ with $\Gamma_1= \Gamma_0\cup \{\alpha_1\}$ and $\A_1$ the resulting hexagon decomposition when adding $\alpha_1$ to $\A'_0$. Observe that the total intersection between arcs in $\A_1$ and arcs in $\A_0$ is still bounded above by $I_1$.

We now repeat the above process to add a second curve and so on. At each step, we might be required to perform flips (at most $D$), which may further increase the intersection between the arcs in $\A_0$ and the arcs $\A'_k$ (the arcs at step $k$), but there is always a bound (denoted $I_k$) on this intersection which only depends on topology and the number of steps. Using the same trick as above, we add a curve $\alpha_k$ which intersects $\A_0$ at most $I_k$ times. The process ends in at most $m\leq \kappa_c(\Sigma)$ steps, resulting in a pants decomposition where each curve intersects the original multiarc at most $I_m$ times.

The following lemma will allow us to conclude.

\begin{lemma}
Consider $H= (\Gamma,\A)$ on a surface $\Sigma$. For $K$ a constant, let $P,Q$ be pants decompositions which both contain $\Gamma$, and such that $\max\{i(P,\A),i(P,\A)\}\leq K$. Then the pants distance between $P$ and $Q$ is bounded above by $R(K)$, a function of the topology of $\Sigma$ and $K$.
\end{lemma}
\begin{proof}
This basically follows from a finiteness argument. Any multicurve on $\Sigma$, disjoint from $\Gamma$, is uniquely determined by its intersection numbers with arcs in $\A$. As we have a bound on these numbers ($K$), there are finitely many such pairs of multicurves, thus finitely many choices for $P$ and $Q$. Among these, there is a pair at maximal distance. This quantity only depends on the topology of $\Sigma$, the topological type of $\Gamma$, and $K$. Taking again the maximum among all topological types of $\Gamma$ (the number of these again only depends on the topology of $\Sigma$), we obtain an upper bound on the possible distance between $P$ and $Q$ as claimed.
\end{proof}

Now, using our upper bound of $I_m$ on the intersection of a pants decomposition with $H= (\Gamma_0,\A_0)$, and applying the above lemma, we can conclude the following:

\begin{proposition} Given $H$, any two possible choices of pants decomposition $P$ and $Q$ as images of $H$ by $\psi$ are at distance at most $R(I_m)$ in $\pS$.
\end{proposition}
\qedsymbol

Now consider $H$ and $H'$ at distance $1$ in $\hS$. If they are related by a flip, by the argument given above, then their images by $\psi$ are at a bounded distance apart, where this distance depends only the topology of $\Sigma$. If they are related by a curve addition, then their images can be chosen to be identical. As such, we proved the following:

\begin{lemma}\label{lem:psidist}
There exists a constant $C_2=C_2(\Sigma)$ that depends only on the topology of $\Sigma$ such that for all $H,H'\in \hS$ we have
$$
d_{\mathcal{P}}(\psi(H), \psi(H')) \leq C_2 \,d_{\mathcal{H}}(H,H').
$$
\end{lemma}
\qedsymbol

Finally, we observe that $\phi$ and $\psi$ are quasi-inverses. Indeed, we have $\psi(\phi(P))=P$ for all $P\in \pS$, and $\phi(\psi(H))$ is at most a constant distance from $H$ (where the constant only depends on topology). Putting these results together, we have now shown that the maps $\phi$ and $\psi$ are quasi-isometries, hence showing Theorem \ref{thm:quasipants}.

\section{The geometric hexagon graph}\label{sec:ghg}

By adding additional data to each hexagon decomposition, we construct a new graph which will give us a quasi-isometric model for the mapping class group. The graph will be associated to a (fixed) hyperbolic metric $X$ which is homeomorphic to $\Sigma$, and denoted $\hX$.

\subsection{Vertices and edges.} For this graph, we require that curves have an orientation. We fix an orientation on each curve (it can be arbitrary).

The vertices of $\hX$ are hexagon decompositions as above, but with an integer (called a weight) prescribed to each curve in the multicurve. Thus precisely, vertices are weighted hexagon decompositions $H= (\Gamma^w,\A)$, where $\Gamma^w=(\Gamma,w)$, that is a multicurve $\Gamma=(\gamma_1,\hdots,\gamma_m)$ with a set of weights $w\in \Z^m$, and $\A$ a maximal multiarc on $X\setminus \Gamma$.

Note that if a curve $\alpha$, with an orientation, is given weight $k$, this corresponds to $\alpha^{-1}$ (that is $\alpha$ with the opposite orientation) with given weight $-k$. In other words, we can think of each un oriented curve as an equivalence class, and $\alpha,k$ as a representative. Finally, weight $0$ is not particular, other than it is the only weight for which $(\alpha,0)$ and $(\alpha^{-1},0)$ correspond to the same class.

Edge relations are defined as before, with the following caveats:
\begin{enumerate}[i)]
\item Let $w, w' \in \Z^m$ differ by only one coordinate that differs by $1$ (expressed otherwise, $|| {w-w'} ||_1 = 1$) then there is an edge between $(\Gamma^w, \A)$ and $(\Gamma^{w'},\A)$. In other words, changing a single weight by $1$ induces an edge.\\
\item We need to define how to associate a weight to an added curve. We explain this below.
\end{enumerate}

{\it The weight of an added curve.} This is where we use the hyperbolic metric $X$.

Let $\alpha$ be compatible to a given hexagon decomposition $\left(\Gamma_w, \calA\right)$, that is if $\alpha \cap \Gamma = \emptyset$ and $i(\alpha, a)\leq 1$ for every arc $a\in \calA$.

Note that all arcs $a \in \calA$ with $i(a,\alpha)=1$ split into two arcs $a_1,a_2$ in $X\setminus \{\Gamma \cup \alpha\}$. These arcs are obtained by surgering $a$ at its \emph{geometric} intersection point with $\alpha$ and let $a_1$ and $a_2$ are the resulting (homotopy classes of) arcs. Both $a_1$ and $a_2$ have exactly one of their endpoints on $\alpha$.

We can equivalently think of the homotopy classes of arcs $a_1$ and $a_2$ as being realized by unique orthogeodesics (using the hyperbolic metric $X$). This gives rise to a converse construction of $a$, starting from the arcs $a_1$ and $a_2$. To begin with, we give $\alpha$ an orientation. We orient $a_1$ so that it ends on $\alpha$, and $a_2$ so that it begins on $\alpha$. The (homotopy class of) arc $a$ is obtained via the orthogeodesics $a_1$ and $a_2$ by first following $a_1$, then following $\alpha$ for a certain length (which we denote by $t_\alpha(a)$) and then by following $a_2$ (see Figure \ref{fig:splittingarc}).

\begin{figure}[htbp]
\leavevmode \SetLabels
\L(.39*.78) $a$\\
\L(.36*.45) $\alpha$\\
\L(.683*.7) $a_1$\\
\L(.61*.27) $a_2$\\
\L(.69*.45) $t_\alpha(a)$\\
\endSetLabels
\begin{center}
\AffixLabels{\centerline{\includegraphics[width=9cm]{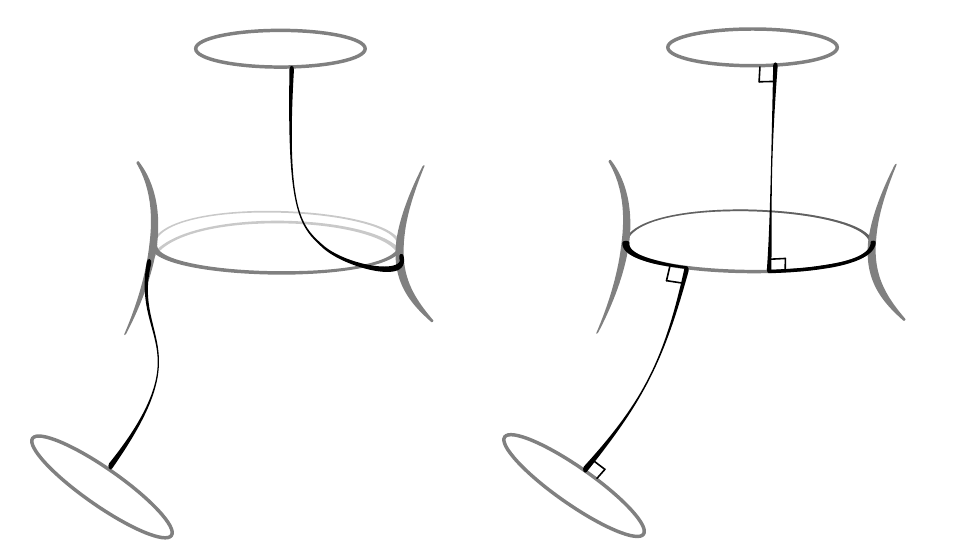}}}
\caption{An arc $a$ and its representation by a concatenation of arcs}
\label{fig:splittingarc}
\end{center}
\end{figure}
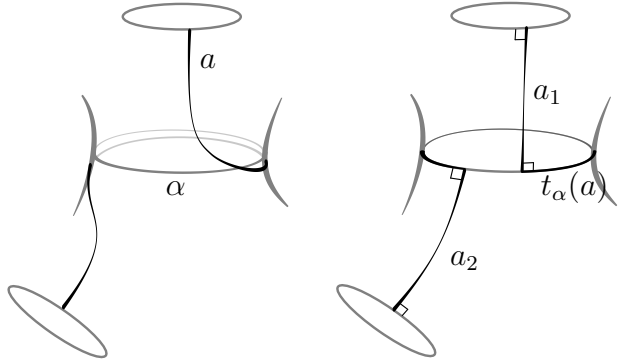

The result is an oriented arc whose unoriented version is $a$. Note that the quantity $t_\alpha(a)$ does not depend on our choice of $a_1$ and $a_2$, but only on the orientation of $\alpha$.

Let $\alpha$ be a compatible curve that we add to $(\Gamma^w, \calA)$.  This results in $\left((\Gamma')^{w'},\calA'\right)$ where:

\begin{enumerate}[i)]
\item $\Gamma'= \Gamma\cup \{\alpha\}$\\
\item $\calA'$ consists of all arcs of $\calA$ that don't intersect $\alpha$, to which we add all arcs of $\calA$ split by $\alpha$.\\
\item The weights of the curves in $\Gamma'$ are defined as follows:
\end{enumerate}

For all curves that were previously in $\Gamma$ the weights remain unchanged.

The weight of $\alpha$ relative to an arc $a$ is defined as
$$
w(\alpha,a)= \{ \lfloor t_{\alpha}(a)/{\ell_{X}(\alpha)}\rfloor\}.
$$
Here, $\ell_{X}(\alpha)$ is the length of the unique closed geodesic freely homotopic to $\alpha$ for the hyperbolic metric $X$, and $t_\alpha(a)$ is defined above.

In order to make the weight independent of the choice of arc, we set
$$
w(\alpha)=\max_{a\in \calA \setminus \calA'}\{ \lfloor t_{\alpha}(a)/{\ell_{X}(\alpha)}\rfloor\}
$$
where $a$ runs over all arcs split by $\alpha$.

Roughly speaking, the weight, associated to a choice of arc $a$, is defined as the (whole) twist parameter (or power) of $\alpha$ necessary to obtain the arc $a$ by concatenating $a_1$ and $a_2$ with a power of $\alpha$. As there are multiple arcs to choose from, we take the maximum value to avoid any ambiguity. This is an arbitrary choice and we could have chosen the minimum for instance. In fact in the next lemma, we show that for any two arcs, the values obtained will differ by at most one.

\begin{lemma}\label{lem:curveadding2}
Weights of arcs differ by at most 1.
\end{lemma}

\begin{proof} Let $a,b$ be disjoint arcs in $\mathcal{A}$ and $\alpha$ a simple closed curve that intersects both $a$ and $b$ once.

As $a$ and $b$ are disjoint, it follows that $t_{\alpha}(a)$ and $t_{\alpha}(b)$ cannot differ by more than $\ell_X(\alpha)$. Otherwise, one of the arcs would wrap (at least) one more time around $\alpha$ than the other, and the arcs would intersect. Hence
\[|t_{\alpha}(a)- t_{\alpha}(b)| \leq \ell_X(\alpha)\]
and then by definition
$$ | w(\alpha,a) - w(\alpha,b) | = \frac{ |t_{\alpha}(a)- t_{\alpha}(b)| }{\ell_X(\alpha)} \leq 1$$
as claimed.
\end{proof}

We finish this subsection with an upper bound on the degree of vertices.

\begin{theorem}\label{thm:valency}
Any vertex in $\hX$ is of finite valency $F$ where $F$ depends only on the topology of $\Sigma$.
\end{theorem}
\begin{proof}
Let $(\Gamma^w, \calA)$ be a vertex. Edges leaving this vertex correspond can correspond to different types of elementary moves:
\begin{enumerate}[i)]
\item Flipping arcs of $\mathcal{A}$, contributing at most $|\mathcal{A}|$ edges.

\item Each weight can change by $\pm 1$, thus contributing $2 |\Gamma|$ edges.

\item Curves can be added, but by Lemma \ref{lem:curveadding1} there are most $2^{|\calA|}$ such curves.

\item Curve removal: this is the part which requires some work. We need to estimate how many different different vertices can result in $(\Gamma^w, \calA)$ after a curve addition.
\end{enumerate}

We consider $\alpha \in \Gamma$, and we look at how it might have appeared under a curve addition.

Note there are possibly an infinite number of hexagon decompositions of $X\setminus\left(\Gamma\setminus\alpha\right)$ that would topologically result in $(\Gamma,\calA)$ by adding $\alpha$ but only finitely many would result in $(\Gamma^w,\calA)$ with the correct weight along $\alpha$ as we now explain.

Any arc in such a hexagon decomposition must have been obtained by concatenating an arc of $\calA$, say $a_1$, with an arc $c$ belonging to $\alpha$, and then another arc of $\calA$, say $a_2$. There are at most $|\calA|$ choices for both $a_1$ and $a_2$, thus at most $|\calA|^2$ choices for the pair. Now because the weight of $\alpha$ is known, there are at most $2$ choices for the arc $c$ (the arc must wrap the prescribed number of times around $\alpha$, or possibly once less following Lemma \ref{lem:curveadding2}).

Thus there are at most $2|\calA|^2$ choices for each arc. The full hexagon decomposition requires at most $N$ new arcs where
$$N=\kappa_{\mathrm{a}}(X\setminus\left(\Gamma\setminus\alpha\right) ).$$
There are at most
$${2|\calA|^2\choose N} $$
choices.
This was for a given choice of $\alpha\in \Gamma$. Multiplying by $|\Gamma|$, which is bounded above by the curve complexity of $X$, completes the upper bound on possible vertices that are related to  $(\Gamma^w, \calA)$ by a curve removal.

The quantity $F$ is the sum of all of these upper bounds.
\end{proof}

\subsection{$\hX$ is a model for the mapping class group}
With these basic properties in hand, we will now be able to show the following.

\begin{theorem}\label{thm:qimcg} Let $X$ be a closed surface with hyperbolic metric. Then the mapping class group of $X$ acts on $\hX$ by quasi-isometries and with finite quotient. As such, $\hX$ is a quasi-model for the mapping class group.
\end{theorem}

To prove this theorem we will use the following lemmata, the first of which is a version of the Milnor-Schwarz lemma for quasi-isometric actions. Although it is not as used as its isometric action counterpart, it is certainly well-known to experts. For a proof, see for instance \cite[Theorem 4.40]{CHT}.  

\begin{lemma}\label{lem:MS}[Milnor-Schwarz]
Let $G$ be a group and $X$ a geodesic metric space on which $G$ acts with $(\tilde K, \tilde C)$ uniform quasi-isometries. Assume that the action is
properly discontinuous and cocompact. Then, $G$ is finitely generated and there are constants $K\geq 1$, $C\geq 0$ such that $G$ is $(K,C)$-quasi-isometric to $X$.
\end{lemma}

Let $\Mod(X)(=\Mod(\Sigma))$ be the mapping class group of $X$ (or of $\Sigma$). The action of the mapping class group on $\calH(\Sigma)$ is straightforward: it send multi-curves to multi-curves and multi-arcs to multi-arcs by homeomorphisms. Note that $(\Gamma,\A)\in \calH(\Sigma)$ always has infinite stabilizers because of Dehn twists along curves in $\Gamma$. The weights on vertices will prevent this from happening in $\hX$.
\begin{lemma} $\Mod(X)$ acts on $\hX$ is by quasi-isometries.
\end{lemma}
\begin{proof}
In addition to the simplicial action on $\calH(\Sigma)$, we need to define how the mapping class group acts on weights. The action we describe is a quasi-action, and we describe it at a vertex $(\Gamma^w,\calA)$ for an individual weighted curve $\alpha$ where $\alpha\in\Gamma^w$. Consider a mapping class $\varphi \in \Mod(X)$. A weight needs to be attributed to $\varphi(\alpha)$ for a homeomorphism $\varphi$. If the homeomorphism is a Dehn twist along $\alpha$, the weight changes by $\pm 1$ depending on the direction of twist and orientation of $\alpha$. More generally, we can compute the weight of the image as follows.

Let $\alpha$ be of weight $k\in \Z$. We consider a pair of disjoint simple arcs $a_1,a_2\in \calA$ with exactly one endpoint on either side of $\alpha$. If no such pair exists, by performing a bounded number of flips (as in the proof of Lemma \ref{lem:gotopants}), one can ensure that this is the case. (Note that in the case of a one-holed torus, no such pair exists, and in that case a similar argument works with a single arc, but we omit the details here. We are doing this to avoid handling the case where $\alpha$ is separating or non-separating for all other surfaces.) Now we consider the arc $a$ obtained by concatenating $a_1$ (so that it ends on $\alpha$), then following $\alpha$ exactly $k$ times around and then if necessary until the endpoint of $a_2$, and then following $a_2$. We now consider its image $\varphi(a)$ by the mapping class, which is an arc that intersects $\varphi(\alpha)$ exactly once. We can now compute the weight of $\varphi(\alpha)$ by cutting $\varphi(a)$ in its intersection point to obtain $a_1'$ and $a_2'$ and measuring the length of the arc $t'$ needed to concatenate to get $a$ (as in Figure \ref{fig:splittingarc}). The weight of $\varphi(\alpha)$ is
$$
\frac{t'}{\ell_X(\varphi(\alpha))}.
$$
We now need to show that this weight does not depend (too much) on the choice of $a_1$ and $a_2$.

The idea is similar to what happens in Lemma \ref{lem:curveadding2}. There were different choices for $a_1$ and $a_2$, but all resulting in arcs that are either disjoint, or possibly intersect at most once. Hence, as in the proof of Lemma \ref{lem:curveadding2}, the weights associated to any such choice differ by at most $\pm 2$. The difference here is that the arcs might intersect one time whereas in the proof of the lemma, they were necessarily disjoint. This shows that the action is (quasi-)well-defined.

Now to show that the mapping class group acts by quasi-isometries, we need to show that two vertices at distance $1$ are sent to vertices a (uniformly) bounded distance apart.
\begin{enumerate}[i)]
\item If two vertices $x,y$ differ by a flip, then $\varphi(x)$ and $\varphi(y)$ also differ (topologically) by a flip. Furthermore, the weights associated to each curve can be computed using the same arcs (up to a single flip), so by the above argument are comparable.
\item If two vertices $x,y$ differ by a curve addition, then the weights of a curve in $\varphi(x)$ and $\varphi(y)$ can be computed using arcs that are either the same or that intersect at most once. Hence, as above, the weights are always comparable.
\end{enumerate}
In all situations, the weights are uniformly comparable, and since the number of curves only depends on the curve complexity of the surface, this shows that if two vertices $x,y$ are distance $d$ apart, then $\varphi(x)$ and $\varphi(y)$ are distance at most $Cd$ apart where $C$ only depends on the topology of the surface.
\end{proof}

The quotient of $\hX$ by $\Mod(X)$ is exactly the quotient of $\hS$ by $\Mod(\Sigma)$, and hence cocompact. This is because any two hexagon decompositions that only differ by weights are related by Dehn twists along their curves, and so correspond to the same point in $\hS$.

By Theorem \ref{thm:valency}, $\hX$ has finite degree, and so the action of $\Mod(X)$ has finite stabilizers. Hence, the action of the mapping class group by quasi-isometries is proper and cocompact, and by Lemma \ref{lem:MS}, Theorem \ref{thm:qimcg} follows.

\end{document}